\documentclass[reqno, final]{article}

\usepackage{amsmath,amsthm,amsfonts,amssymb}

\usepackage[english]{babel}

\usepackage[utf8x]{inputenc}

\usepackage{graphicx}

\date{}

\usepackage[a4paper]{geometry}
\usepackage{mathtools}

\usepackage{xcolor}

\newcommand{\ee}{\varepsilon}
\newcommand{\pa}{\partial}

\newtheorem{theorem}{Theorem}
\newtheorem{lemma}{Lemma}

\theoremstyle{definition}
\newtheorem{remark}{Remark}

\newcommand{\uo}{u}

\newcommand{\Pee}[1]{P_\ee\left(#1\right)}

\usepackage[hidelinks]{hyperref}
\usepackage[noabbrev,nameinlink]{cleveref}

\newcommand{\email}[1]{\texttt{\href{mailto:#1}{#1}}}

\begin{document}

\title{
A nonlocal memory strange term arising in the \\
critical scale homogenisation of
a diffusion equation with a\\
dynamic boundary condition
}

\author{
	J.I. Díaz\thanks{Instituto de Matemática Interdisciplinar, Universidad Complutense de Madrid. \email{jidiaz@ucm.es}} 
	\and 
	D. Gómez-Castro\thanks{Instituto de Matemática Interdisciplinar, Universidad Complutense de Madrid. \email{dgcastro@ucm.es}} 
	\and 
	T.A. Shaposhnikova\thanks{Faculty of Mechanics and Mathematics. Moscow State University. \email{shaposh.tan@mail.ru}} 
	\and 
	M.N. Zubova\thanks{Faculty of Mechanics and Mathematics. Moscow State University. \email{zubovnv@mail.ru}}
}

\maketitle

\begin{abstract}
	Our main interest in this paper is the study of homogenised limit of a parabolic equation with a nonlinear dynamic boundary condition of the micro-scale model set on a domain with periodically place particles. 
	We focus on the case of particles (or holes) of critical diameter with respect to the period of the structure. 
	Our main result proves the weak convergence of the sequence of solutions of the original problem to the solution of a reaction-diffusion parabolic problem containing a ``strange term''. 
	The novelty of our result is that this term is a nonlocal memory solving an ODE. 
	We prove that the resulting system satisfies a comparison principle.
	
	
\end{abstract}

\tableofcontents

\section{Introduction and statement of results}
A well-known effect in our days in homogenisation theory is the appearance
of some changes in the structural modelling of the homogenised problem for
suitable critical size of the elements configuring the ``micro-structured'' medium which exhibits small-scale spatial heterogeneities or obstacles (also denoted as \textit{particles} in the context of Chemical Engineering). From the mathematical view point a first result was due to V. Marchenko and E. Hruslov \cite{Marchenko}. 
The attention on this effect considerably increased after the presentation of the appearance of some ``strange terms'' due to D. Cioranescu and F.Murat \cite{Ci-Mu}. 
Both articles dealt with linear equations with Neumann and Dirichlet boundary conditions, respectively. In many other papers on critically homogenisation problems the modelling of the reaction kinetics at the micro or nano scales is given by a nonlinear Robin type boundary condition on the surface of the chemical particles, complemented by a pure diffusion equation in the exterior spatial domain to them. 
It is impossible to mention all of them here (we send the reader to the papers on the homogenisation of the problems with classical boundary conditions of the Robin type, including the nonlinear Robin type condition \cite{Jager11,ZSh,ShZ,GLPSHZ,6} and the bibliographic exposition in our previous paper \cite{DGPS}): obviously, the nature of this ``strange term'' may be completely different according to the peculiarities of the formulation in consideration.\\

In this paper we shall consider some dynamic problems in which, depending on suitable characteristic scales, the surface reaction on the boundary of the particle is also dynamic and so, its formulation in terms of Robin type boundary conditions must be modified. 
We recall that the modelling of many different problems involving dynamic boundary conditions is very natural in many different areas and that its mathematical treatment attracted the attention of very distinguished authors since the beginning of the past century. 
A quite complete list of references dealing with nonlinear problems with dynamic boundary conditions, starting already in 1901, can be found, e.g., in the survey papers \cite{BeDiVr} and \cite{Bandle-Reichel}. 
The partial differential equation is sometimes an elliptic equation (and thus there is a great contrast between a stationary interior law and a dynamic boundary condition). 
Nevertheless, the dynamic boundary condition may coexists with a parabolic equation (linear or not). For some recent references\ see, e.g.,\cite{Amman-Fila,Arrieta-Quitner-Bernal,Escher,Bandle-Reichel} and \cite{Vazquez-Vitillaro}.\\

As said before, our main interest in this paper concerns the modification of
the homogenised equation with respect to the nonlinear terms involved in the micro-scale. For the sake of simplicity in the presentation we shall
consider here only the case of a linear surface reaction term but it seems possible to adapt our techniques of proof to the consideration of quite general nonlinear reactions terms as in our paper \cite{DGPS}.\\

To be more precise, as usual, the heterogeneity scale is assumed to be much smaller than the macroscopic scale and that the microscopic heterogeneities (particles or holes) are periodically placed in the spatial domain giving rise to a parameter $\epsilon \rightarrow 0$. 
In fact we work on the spatial domain $\Omega _{\varepsilon }$, obtained by removing $G_{\epsilon }$, a collection of small particles. \\

More specifically, let $\Omega$ be a bounded domain in $\mathbb{R}^{n}$, $n\ge 3$, with smooth boundary $\pa\Omega$. We denote the unit cube by $Y=(-1/2,1/2)^{n}$. Let 
\begin{equation*} 
	G_{0}=\{x:|x|<1\}. 
\end{equation*} 
for $\delta > 0$ and $B \subset \mathbb R^n$ we denote by $\delta{B}=\{x|\delta^{-1}x\in B\}$.
For a positive parameter ${\ee} > 0$ we introduce the domain
$$
\widetilde{\Omega}_{\ee}=\{x\in\Omega \mid d(x,\partial\Omega)>2{\ee}\}.\\
$$
We set
$$
G_{\ee}=\bigcup_{j\in \Upsilon_{\ee}}(a_{\ee}G_{0}+{\ee}j)=\bigcup_{j\in\Upsilon_{\ee}}G^{j}_{\ee},\\
$$
where $\Upsilon_{\ee}=\{j\in \mathbb{Z}^{n}:\overline{G^{j}_{\ee}}\subset {\ee}Y+{\ee}j,\,G^{j}_{\ee}\cap\overline{\Omega}_{\ee}\ne \emptyset\}$, $|\Upsilon_{\ee}|\cong d{\ee}^{-n}$, $d=const>0$, $\mathbb{Z}^{n}$ is the set of vectors with integer coordinates, $a_{\ee}=C_{0}{\ee}^{\gamma}$ is the radius of the particles (or perforations).
We denote by $P^{j}_{\ee}$ the center of the cell of periodicity $Y^{j}_{\ee}$. Let us note that
$$
\overline{G^{j}_{\ee}}\subset T^{j}_{{\ee}/4}\subset Y^{j}_{\ee}, \\
$$
where $T^{j}_{\rho}$ is the ball with the center at the point $P^{j}_{\ee}$ and with radius $\rho$.
Finally, we define the sets
$$
\Omega_{\ee}=\Omega\setminus\overline{G_{\ee}},\,\,S_{\ee}=\pa{G_{\ee}},\,\,\pa\Omega_{\ee}=\pa\Omega\cup S_{\ee}.\\
$$
In $Q_{\ee}^{T}=\Omega_{\ee}\times (0,T)$ we consider the following parabolic problem with a dynamical boundary condition
\begin{equation}
\begin{dcases}
 \alpha \frac{ \partial u_{\ee}}{\partial t}-\Delta{u_{\ee}}=f(x,t), & (x,t)\in Q^{T}_{\ee},\\
 {\ee}^{-\gamma}  \beta \frac{ \partial u_{\ee}}{\partial t}+\pa_{\nu}u_{\ee}+{\ee}^{-\gamma}\lambda{u_{\ee}}={\ee}^{-\gamma}g(x,t), & (x,t)\in S_{\ee}^{T}=S_{\ee}\times (0,T),\\
u_{\ee}=0, & (x,t)\in \Gamma^{T}=\pa\Omega\times (0,T),\\
\alpha u_{\ee}(x,0)=0, & x\in \Omega_{\ee}, \\
\beta u_{\ee}(x,0)=0, & x\in S_\ee,
\end{dcases}
\label{eq:1}
\end{equation}
where $ \alpha,  \beta \ge 0$,  $Q^{T}=\Omega\times (0,T)$, $\lambda>0$ is constant, $\nu$ is the unit outward normal vector to the boundary of the cylinder $Q^{T}_{\ee}$, $g\in \mathcal C^1(\overline {Q^T})$ (for the sake of simplicity of the exposition), 
$$\gamma=\frac{n}{n-2} \qquad \textrm{and} \qquad n\ge 3 $$
and, either
\begin{equation}
	\tag{F$_1$}
	\alpha > 0 \textrm{ and } f\in L^{2}(Q^{T})
\end{equation}
or
\begin{equation}
	\tag{F$_2$}
	\beta > 0 \textrm{ and } f\in H^1(0,T;L^2(\Omega))
\end{equation}

We point out that the linear dynamic boundary condition contains a parameter ${\varepsilon }^{-\gamma }$, where $\gamma $ has the critical value, on the boundary of particles of the critical size.\\

In previous papers on the homogenisation of the problems in perforated domains with dynamic boundary conditions 
(e.g.\ \cite{Timofte,Timofte2008,Angulano,Wang}) the diameter of the particles (or holes) was assumed of the same order as a period of the structure. 
As consequence the homogenised reaction term (now appearing in the interior of the whole domain $\Omega $) preserved the same structure assumption than the surface reaction term in the micro-model formulation. 
That was in consonance with many other studies on reaction-diffusion problems (see, e.g.\ \cite{CDLT} and its references).\\

Our main goal in this paper is to prove the appearance of a
``strange term'' in the effective parabolic
problem and to characterise it in terms of the surface reaction term than of
the micro-model formulation. As we shall see, this new term appears even if
there is no surface reaction term in the micro-model formulation (i.e. for $\lambda =0$). Our main result in this paper proves the weak convergence of
the sequence of (the extension of) solutions of the original problems to the
solution of the following homogenized problem, as ${\varepsilon }\rightarrow
0$:
\begin{theorem}
\label{thm:main}
Let $n\ge 3$, $\gamma=\frac{n}{n-2}$ and let $u_{\ee}$ be the unique weak solution of the problem \eqref{eq:1}.
Then, there exists an extension $\widetilde {u_\ee} \in L^2(0,T; H_0^1(\Omega))$ of $u_\ee$ and function $\uo \in L^2(0,T; H_0^1(\Omega))$ such that 
\begin{subequations} 
\label{eq:convergence}
\begin{align}
	\widetilde{u_{\ee}}&\rightharpoonup \uo \quad \textrm{weakly in } L^{2}(0,T;H^{1}_{0}(\Omega)),
		\label{eq:25}  \ \\
	\pa_{t}\widetilde{u_{\ee}}&\rightharpoonup \pa_{t}\uo\quad \textrm{weakly in } L^{2}(0,T; L^{2}(\Omega)),
		\label{eq:26}  \ \\
	\widetilde{u_{\ee}}&\to \uo\quad \textrm{strongly in } L^{2}(0,T; L^{2}(\Omega)).
		\label{eq:27}  \ 
\end{align}
\end{subequations}
This limit function $\uo$ is the unique weak solution of the system
\begin{equation}
	\label{eq:28}
		\begin{dcases}
			 \alpha \frac{\partial \uo}{\partial t} - \Delta \uo + (n-2)C_{0}^{n-2}\omega_{n} H = f & Q^T,\\
			 \beta \frac{\partial H}{\partial t} +  \frac{n-2}{C_0} H = \lambda (\uo - H) +  \beta \frac{\partial \uo}{\partial t} - g & Q^T, \\
			\uo = 0 & \partial \Omega \times (0,+\infty), \\
			\alpha \uo(x,0) = 0 & \Omega,\\
			\beta H(x,0) = 0 & \Omega.
		\end{dcases}
	\end{equation}
\end{theorem}
System \eqref{eq:28} is not a standard parabolic problem (since there is no diffusion term for $H$). Nevertheless,  there are some systems in the literature  keeping several common points with such a  system. See, for instance \cite{DiSt} and \cite{Chipot}.  \\

Notice that, when $\beta = 0$, we recover the known equation for the 
\emph{strange term} in the elliptic ($\alpha = 0$) and parabolic ($\alpha > 0$) cases (see \cite{Gon,DGPS,Jager10}) 
\begin{equation}
	\label{eq:defn H when beta 0}
	\frac{n-2}{C_{0}}H=\lambda (\uo-H)-g.
\end{equation}
Moreover, since the equation for $H$ contains a term $\partial \uo/\partial t$,
it seems natural to use the change of variable
\begin{equation}
v=\uo-H.
\end{equation}%
Hence system \eqref{eq:28} can be equivalently written as 
\begin{equation}
	\label{eq:28b} \tag{\ref{eq:28}b}
	\begin{dcases}
		 \alpha \frac{\partial \uo}{\partial t} - \Delta \uo + (n-2)C_{0}^{n-2}\omega_{n} (\uo - v) = f & Q^T,\\ 
		 \beta \frac{\partial v}{\partial t} + \left (\frac{n-2}{C_0} + \lambda \right) v = \frac{n-2}{C_0} \uo + g & Q^T \\
		\uo = 0 & \partial \Omega \times (0,T), \\
		\alpha \uo(x,0) = 0 & \Omega,\\
		\beta v(x,0) = 0 & \Omega.
	\end{dcases} 
\end{equation}
We will prove in \Cref{sec:uniqueness of the limit problem} that it has a unique weak solution.
Furthermore, if $f,g\geq 0$ then $u, v\geq 0$ and, hence $H \le u$. 

In formulation \eqref{eq:28b} we can solve the first order ODE for $v$ explicitly, and solving for $H$ we obtain, for $\beta > 0$ 
\[
	H(x,t)=\uo(x,t)-\frac{1}{ \beta}\int\limits_{0}^{t}\left( \frac{n-2}{C_{0}}\uo(x,s)+g(x,s)\right) e^{-\frac{\lambda +\frac{n-2}{C_{0}}}{ \beta}(t-s)}ds.
\]
Thus we conclude that in the case of a dynamic boundary term the ``strange term'' is given by a \textit{nonlocal memory term} (even if $\lambda =0$). We recall that the comparison principle is not always satisfied in the presence of general nonlocal memory terms.

It is surprising that, when $\alpha = 0$ and $\beta > 0$ the limit obtained in \Cref{thm:main} becomes an elliptic linear Dirichlet boundary value problem depending of the time (as parameter) and with a linear but nonlocal reaction term:
\begin{equation*}
\begin{dcases}
	-\Delta u  + (n-2)C_0^{n-2}\omega_n 
	\left( \uo(x,t)-\frac{1}{\beta}\int\limits_{0}^{t}\left( \frac{n-2}{C_{0}}\uo(x,s)+g(x,s)\right) e^{-\frac{\lambda +\frac{n-2}{C_{0}}}{ \beta}(t-s)}ds \right) \\
	\qquad = f(x,t)   \textrm{ in }\Omega\times(0,T), \\
	u = 0 \textrm{ on } \partial \Omega \times (0,T), 
\end{dcases}
\end{equation*}

The proof of the main result is presented in the next section which we structured by means of several subsections. The last subsection contains the proof of the comparison principle for the parabolic homogenised system (which, in particular, implies the uniqueness of solutions).


\section{Proof of \Cref{thm:main}} 

 The proof applies Tartar's method of oscillating functions that has been successful in the past for the critical case (see, e.g. \cite{Tartar,DGPS}), but introducing some new ideas to deal with dynamical boundary conditions.
 
 \subsection{Existence, uniqueness and convergence of solutions of problem \eqref{eq:1}}

A weak solution of the problem \eqref{eq:1} is defined as a function 
\begin{equation*}
	u_{\ee}\in L^{2}(0,T; H^{1}(\Omega_{\ee},\pa\Omega)),	
	\quad \alpha \pa_{t}u_{\ee}\in L^{2}(0,T; {H^{-1}(\Omega_{\ee})}) \quad \textrm{and} \quad  \beta \pa_{t}u_{\ee}\in L^{2}(0,T; {H^{-1/2}(S_{\ee})}) 	
\end{equation*}
such that $\alpha u(x,0) = 0$ on $\Omega_\ee$ and $\beta  u(x,0) = 0$ on $S_\ee$ 
satisfying the integral identity
\begin{align} 
 \alpha\int\limits_{0}^{T}&\langle \pa_{t}u_{\ee},\phi \rangle _{\Omega_{\ee}}dt
+ \beta{\ee}^{-\gamma}\int\limits_{0}^{T}\langle\pa_{t}u_{\ee},\phi\rangle_{S_{\ee}}\, dt+
\int\limits_{Q_{\ee}^{T}}\nabla{u_{\ee}}\nabla{\phi}dx{dt}+\lambda{\ee}^{-\gamma}\int\limits_{S_{\ee}^{T}}u_{\ee}{\phi}ds\,dt=\nonumber \\
\label{eq:2}
&={\ee}^{-\gamma}\int\limits_{S_{\ee}^{T}}g(x,t)\phi(x,t)ds\,dt+\int\limits_{Q_{\ee}^{T}}f\phi \, dx\,{dt},
\end{align}
where $\phi$ is an arbitrary function from $L^{2}(0,T; H^{1}(\Omega_{\ee},\pa\Omega))$, $\langle \cdot,\cdot \rangle _{\Omega_{\ee}}$ denotes the duality product between $H^{-1}(\Omega_{\ee},\pa\Omega)$ and $H^{1}(\Omega_{\ee},\pa\Omega)$ and $\langle \cdot, \cdot \rangle_{S_{\ee}}$ denotes the duality product between $H^{-1/2}(S_{\ee})$ and $H^{1/2}(S_{\ee})$.
The space $H^{1}(\Omega_{\ee},\pa\Omega)$ is defined as the closure in $H^{1}(\Omega_{\ee})$ of the space of functions infinitely differentiable in $\overline{\Omega_{\ee}}$ and vanishing in a neighbourhood of the boundary $\pa\Omega$. 

\begin{remark}
	We recall that initial data are given in $\Omega_{\ee}$ if $\alpha > 0$ and on $S_\ee$ if $\beta > 0$. The problem has a semigroup solution even if the initial data on $S_\ee$ is not the trace of the data in $\Omega_{\ee}$. However, when this properties hold, solutions are smoother.
\end{remark} 

The existence and uniqueness of solutions to problem \eqref{eq:1} is consequence of well-known results (see, e.g. Esher \cite{Escher}). It is also possible to apply the theory of monotone operators (see \cite{BeDiVr}) or Galerkin's approximation arguments (see \cite{Angulano,Jager11}).
We recall that the above mentioned references show a greater regularity on the time derivative. Thus, by using the time derivatives of $u$ and of its trace as test functions we arrive to the following result, the proof of which is an easy consequence of the above mentioned results:
\begin{theorem}
Problem \eqref{eq:1} has a unique weak solution $u_{\ee}$ and the following estimate holds
\begin{align}
	\| u_\ee \|_{H^1 (Q^T_\ee)} \le K,
\label{eq:3}
\end{align}
where $K$ here and below is a positive constant that does not depend on ${\ee}$.
\end{theorem}

\begin{remark}
	Notice that, when $\alpha = 0$, we require greater regularity of $f$ in order to work more easily with $\frac{\partial u}{\partial t}$. We guess that this technical assumption could be improved by suitable approximation arguments but we shall not enter into the details here. 
\end{remark}

\subsection{Extension  and existence of a limit} 
There exists a uniformly bounded family of extension operators 
\begin{eqnarray*}
	P_{\ee} : H^1 (Q_\ee^T) &\longrightarrow & H^1(Q^T).
\end{eqnarray*}
which, furthermore, preserve the boundary conditions 
\begin{eqnarray*}
	P_{\ee} : H^1 (Q_\ee^T, \Gamma^T) &\longrightarrow & H^1(Q^T, \Gamma^T).	
\end{eqnarray*}
where $\Gamma^T = \big(\partial \Omega \times (0,T) \big) \cup \big(\Omega \times \{0\} \big)$. 
See, e.g., \cite{Ci-Saint,OS95}. 
Hence
\begin{align}
	\Vert\Pee{u_{\ee}}\Vert_{H^{1}(Q^{T})}&\le {C} \Vert{u_{\ee}}\Vert_{H^{1}(Q^{T}_{\ee})},
		\label{eq:23} 
\end{align}

Estimate \eqref{eq:23} implies that there exists a subsequence (we preserve for it the notation of the original sequence) such that, as ${\ee}\to 0$, we have \eqref{eq:convergence}.
	
\subsection{Constructing a functional inequality.}
Due to the weak formulation of \eqref{eq:1} 
and using the monotonicity of the involved vectorial operator, as in \cite{DGPS}, we can use a very weak formulation of the problem leading to the new inequality
\begin{align}
 \alpha \int\limits_{0}^{T}\int\limits_{\Omega_{\ee}}\pa_{t}\phi(\phi-u_{\ee})dx{dt}&+ \beta {\ee}^{-\gamma}\int\limits_{0}^{T}\int\limits_{S_{\ee}}\pa_{t}\phi(\phi-u_{\ee})ds{dt}+\nonumber \\
&\qquad +\int\limits_{0}^{T}\int\limits_{\Omega_{\ee}}\nabla\phi\nabla(\phi-u_{\ee})dx{dt}+{\ee}^{-\gamma}\int\limits_{0}^{T}\int\limits_{S_{\ee}}\lambda\phi(\phi-u_{\ee})ds{dt}
\label{eq:29}\ \\
&\ge {\ee}^{-\gamma}\int\limits_{0}^{T}\int\limits_{S_{\ee}}{g(x,t)}(\phi-u_{\ee})ds{dt}+\int\limits_{0}^{T}\int\limits_{\Omega_{\ee}}f(\phi-u_{\ee})dx{dt}-\nonumber\\
&\qquad -\frac{\alpha}{2}\Vert{\phi(x,0)}\Vert^{2}_{L^{2}(\Omega_{\ee})}
	-\frac{\beta}{2}{\ee}^{-\gamma}\Vert{\phi(x,0)}\Vert^{2}_{L^{2}(S_{\ee})},\nonumber
\end{align}
where $\phi(x,t)=\psi(x)\eta(t)$, $\psi\in H^{1}_{0}(\Omega)$, $\eta\in C^{1}[0,T]$.

\subsection{Selection of the oscillating test function: spatial component} 
We will select an oscillating test function $\phi_\ee = \phi - W_\ee (x) H (\phi)$. Function $W_\ee$ is our usual choice that allows to change the study of boundary integrals over $G_\ee$ to a union of \emph{large} balls
\begin{equation*}
	T_\ee = \bigcup_{j \in \Upsilon_{\ee}} T_{\ee/4}^j
\end{equation*}
where $T_{\ee/4}^j$ is the ball of radius $\ee/4$ centered at $\ee j$.
We introduce the function $w^{j}_{\ee}(x)$ as a solution of the following problem
\begin{equation}
	\begin{dcases} 
	\Delta{w^{j}_{\ee}}=0, & x\in T^{j}_{{\ee}/4}\setminus \overline{G^{j}_{\ee}},\\
	 w^{j}_{\ee}=1, & x\in \pa{G^{j}_{\ee}},\\
	 w^{j}_{\ee}=0, &x\in \pa{T^{j}_{{\ee}/4}}.
	\end{dcases}
	\label{eq:32}\ 
\end{equation}
For a ball it is known that
\begin{equation}
	w^{j}_{\ee}(x)=\frac{|x|^{2-n}-(\frac{{\ee}}{4})^{2-n}}{a_{\ee}^{2-n}-(\frac{\ee}{4})^{2-n}}
	\label{eq:41}\ 
\end{equation}
is the explicit solution. We set
$$
	W_{\ee}(x)=
	\begin{dcases}
	w^{j}_{\ee}(x), & x\in T^{j}_{{\ee}/4}\setminus\overline{G^{j}_{\ee}},\,j\in \Upsilon_{\ee},\\
	1, & x\in G^{j}_{\ee},\,j\in \Upsilon_{\ee},\\
	0, & x\in \Omega\setminus \overline T_\ee .\\
	\end{dcases}
$$
It is easy to see that $W_{\ee}\in H^{1}_{0}(\Omega)$ and, as ${\ee}\to 0$,
\begin{equation}
	W_{\ee}\rightharpoonup 0\,\,\,\mbox{weakly in }\,\,H^{1}_{0}(\Omega).
	\label{eq:33}\ 
\end{equation}
\subsection{Selection of the oscillating test function: time component} For an arbitrary function $\eta(t)\in C^{1}[0,T]$ and $\psi \in H_0^1 (\Omega)$, let us introduce functions $H_{\ee}^j(t)$, $(j\in \Upsilon_{\ee})$ as a solution
of the following Cauchy problem
\begin{equation}
\begin{dcases}
 \beta \frac{dH_\ee^j}{dt} +\frac{n-2}{C_{0}}H_\ee^j-\lambda\Big(\psi(P_\ee^j)\eta(t)-H_\ee^j\Big)= \beta \psi (P_\ee^j)\frac{d\eta}{dt}-g(P^{j}_{\ee},t),\\
\beta H_\ee^j(0)= \beta \psi(P_\ee^j) \eta(0).\\
\end{dcases}
\label{eq:30}\ 
\end{equation}
The choice of problem may appear arbitrary, but it is precisely so that \eqref{eq:reason for appearance of the strange term} vanishes. Notice that, in particular, 
\begin{equation}
H_\ee^j (t) = H_{\psi \eta} (P_\ee^j, t) 
\end{equation}
where, for $\phi$ smooth, $H_{\phi}$ is
the unique solution of
\begin{equation*}
\begin{dcases}
 \beta \frac{\partial H_{\phi}}{\partial t} +\frac{n-2}{C_{0}}H_{\phi}-\lambda\Big(\phi-H_{\phi}\Big)= \beta \frac{\partial \phi}{\partial t}-g & Q^T, \\
	\beta H_{\phi}(x,0)= \beta \phi(x,0) & \Omega.\\
\end{dcases}
\end{equation*}
When $\beta > 0$, the solution of this problem is given explicitly by
\begin{gather}
H_{\phi}(x,t)=\phi(x,t)-\frac{n-2}{ \beta C_{0}}\int\limits_{0}^{t}e^{-\frac{\lambda+\frac{n-2}{C_{0}}}{ \beta}(t-s)} \Big (\phi(x,s) + g(x,s) \Big) ds.
\label{eq:46}\ 
\end{gather} 
When $\beta = 0$ we can solve directly to obtain $H_\phi = (\lambda \phi - g) /( \frac{n-2}{C_0} + \lambda )$.
Also, we have that
\begin{equation*}
{\beta}\frac{d}{dt} \| H_\phi (t) \|_{L^2 (\Omega)}^2 + \left( \lambda + \frac{n-2}{C_0}  \right) \| H_\phi (t) \|_{L^2 (\Omega)}^2 \le  \left( \| \phi (t) \|_{L^2 (\Omega)} + \left\| \frac{ \partial \phi }{\partial t} (t) \right\| + \| g \|_{L^2 (\Omega)}  \right) \| H_\phi (t) \|_{L^2 (\Omega)}
\end{equation*}
Hence
\begin{equation}
\label{eq:estimate for H phi}
\| H_\phi \|_{L^2 (Q^T)} \le {C} \left( \| \phi(\cdot, 0) \|_{L^2 (\Omega)} + \| \phi\|_{L^2 (Q^T)} + \left\| \frac{ \partial \phi }{\partial t} \right\|_{L^2 (Q^T)} + \| g \|_{L^2 (Q^T)} \right) .
\end{equation}

\subsection{The oscillating test function in space-time} 
Let us define the function
\begin{equation}
	w_{\ee}(\psi\eta)=
	\begin{dcases}
	w^{j}_{\ee}(x)H_\ee^j(x,t), & x\in T^{j}_{{\ee}/4}\setminus \overline{G^{j}_{\ee}},j\in \Upsilon_{\ee},t\in [0,T],\\
	0, & x\in \Omega\setminus\overline{T_\ee}, t\in [0,T].\\
	\end{dcases}
	\label{eq:34}\ 
\end{equation}

We have $w_{\ee}(\psi\eta)\in H^{1}(Q^{T}_{\ee})$ and if we denote by $\Pee{w_{\ee}(\psi\eta)}$ the $H^{1}$ - extension on $Q^{T}$ of the function $w_{\ee}(\psi\eta)$,
satisfying the estimates similar to \cref{eq:23}, we obtain
using \eqref{eq:33} as ${\ee}\to 0$
\begin{gather}
\Pee{w_{\ee}(\psi\eta)}\rightharpoonup 0\,\,\,\mbox{weakly in}\,\,\,L^{2}(0,T;H^{1}_{0}(\Omega)),
\label{eq:35}\ \\
\Pee{w_{\ee}(\psi\eta)}\to 0,\,\,\pa_{t}\Pee{w_{\ee}(\psi\eta)}\to 0\,\,\,\mbox{strongly in}\,\,\,L^{2}(Q^{T}).\label{eq:36}\ 
\end{gather}
Let us take as a test function in the inequality \eqref{eq:29} $\phi(x,t)=\psi(x)\eta(t)-{w_{\ee}(\psi\eta)}$, where $\psi\in C^{\infty}_{0}(\Omega)$, $\eta\in C^{1}[0,T]$. We get
\begin{align}
\int\limits_{0}^{T}&\int\limits_{\Omega_{\ee}} \alpha( \psi(x)\frac{d \eta}{dt}(t)-\pa_{t}w_{\ee}(\psi\eta))(\psi(x)\eta(t)-w_{\ee}(\psi\eta)-u_{\ee})dx{dt}+\nonumber\\
&\qquad +{\ee}^{-\gamma}\sum\limits_{j\in \Upsilon_{\ee}}\int\limits_{0}^{T}\int\limits_{\pa{G^{j}_{\ee}}} \beta(\psi(x)\frac{d \eta}{dt}(t)-\frac{d H_\ee^j}{dt}(t))(\psi(x)\eta(t)-H_\ee^j(t)-u_{\ee})ds{dt}+\nonumber\\
&\qquad+{\ee}^{-\gamma}\sum\limits_{j\in \Upsilon_{\ee}}\int\limits_{0}^{T}\int\limits_{\pa{G^{j}_{\ee}}}\lambda(\psi(x)\eta(t)-H_\ee^j(t))(\psi(x)\eta(t)-H_\ee^j(t)-u_{\ee})ds{dt}+\nonumber\\
&\qquad+ \int\limits_{0}^{T}\int\limits_{\Omega_{\ee}}(\eta\nabla\psi(x)-\nabla{w_{\ee}(\psi\eta)})(\eta(t)\nabla\psi(x)-\nabla{w_{\ee}(\psi\eta)}-\nabla{u_{\ee}})dx{dt}\ge\nonumber\\
&\ge \int\limits_{0}^{T}\int\limits_{\Omega_{\ee}}f(x,t)(\psi(x)\eta(t)-w_{\ee}(\psi\eta))dx{dt}+\nonumber\\
&\qquad+{\ee}^{-\gamma}\sum\limits_{j\in \Upsilon_{\ee}}\int\limits_{0}^{T}\int\limits_{\pa{G^{j}_{\ee}}}{g(x,t)}(\psi(x)\eta(t)-H_\ee^j(t)-u_{\ee})ds{dt}
\label{eq:37}\\
&\qquad-\frac{\alpha}{2}\Vert\psi(x)\eta(0)-{w_{\ee}(\psi\eta)}|_{t=0}\Vert^{2}_{L^{2}(\Omega_{\ee})}
-\frac{\beta}{2}{\ee}^{-\gamma}\Vert\psi(x)\eta(0)-{w_{\ee}(\psi\eta)}|_{t=0}\Vert^{2}_{L^{2}(S_{\ee})}.\nonumber
\end{align}

Taking into account \eqref{eq:35}, \eqref{eq:36}, we conclude
\begin{align}
\lim\limits_{{\ee}\to 0}&\int_{Q^{T}_{{\ee}}}(\psi(x)\frac{d \eta}{dt}(t)-\pa_{t}w_{\ee}(\psi\eta))(\psi\eta-w_{\ee}(\psi\eta)-u_{\ee})dx{dt} \nonumber \\
&=\int\limits_{Q^{T}}\psi\frac{d \eta}{dt}(t)(\psi\eta-\uo)dx{dt},
\label{eq:38}  \\
\lim\limits_{{\ee}\to 0}&\int\limits_{Q^{T}_{\ee}}\nabla(\psi(x)\eta(t))(\nabla(\psi(x)\eta(t))-\nabla{w_{\ee}(\psi\eta)}-\nabla{u_{\ee}})dx{dt} \nonumber \\
&=\int\limits_{Q^{T}}\nabla(\psi(x)\eta(t))\nabla(\psi(x)\eta(t)-\uo)dx{dt}.
\label{eq:39} 
\end{align}
On the other hand, we have
\begin{align}
-\int\limits_{0}^{T}&\int\limits_{\Omega_{\ee}}\nabla{w_{\ee}(\psi\eta)}(\nabla(\psi(x)\eta(t))-\nabla{w_{\ee}(\psi\eta)}-\nabla{u_{\ee}})dx{dt}= \nonumber \\
&=-\sum\limits_{j\in \Upsilon_{\ee}}\int\limits_{0}^{T}\int\limits_{T^{j}_{{\ee}/4}\setminus\overline{G^{j}_{\ee}}}\nabla{w}^{j}_{\ee}\nabla(H_\ee^j(t)(\psi(x)\eta(t)-w^{j}_{\ee}(x)H_\ee^j(t)-u_{\ee}))dx{dt}= \nonumber \\
&=-\sum\limits_{j\in \Upsilon_{\ee}}\int\limits_{0}^{T}\int\limits_{\pa{T^{j}_{{\ee}/4}}}\pa_{\nu}w^{j}_{\ee}H_\ee^j(t)(\psi(x)\eta(t)-u_{\ee})ds{dt}-
\label{eq:40}\  \\
&\qquad -\sum\limits_{j\in \Upsilon_{\ee}}\int\limits_{0}^{T}\int\limits_{\pa{G^{j}_{\ee}}}\pa_{\nu}w^{j}_{\ee}H_\ee^j(t)(\psi(x)\eta(t)-H_\ee^j(t)-u_{\ee})ds{dt}. \nonumber
\end{align}
It is easy to see that
\begin{gather}
\pa_{\nu}w^{j}_{\ee}\Bigl|_{\pa{T^{j}_{{\ee}/4}}}=\frac{(2-n)C_{0}^{n-2}4^{n-1}{\ee}}{1-\alpha_{\ee}};\,\,\,
\pa_{\nu}w^{j}_{\ee}\Bigl|_{\pa{G^{j}_{\ee}}}=-\frac{(n-2)C_{0}^{-1}{\ee}^{-\gamma}}{1-\alpha_{\ee}},
\label{eq:42}\ 
\end{gather}
where $\alpha_{\ee}\to 0$ as ${\ee}\to 0$. Considering the integrals over $S_{\ee}\times (0,T)$. Using \eqref{eq:37}, \eqref{eq:40}-\eqref{eq:42} we obtain
\begin{align}
& \beta {\ee}^{-\gamma}\sum\limits_{j\in \Upsilon_{\ee}}\int\limits_{0}^{T}\int\limits_{\pa{G^{j}_{\ee}}}\left(\psi(x)\frac{d \eta}{dt}(t)-\frac{d H_\ee^j}{dt}(t)\right)(\psi\eta-H_\ee^j(t)-u_{\ee})ds{dt}+ \nonumber \\
&\qquad +\lambda{\ee}^{-\gamma}\sum\limits_{j\in \Upsilon_{\ee}}\int\limits_{0}^{T}\int\limits_{\pa{G^{j}_{\ee}}}(\psi(x)\eta(t)-H_\ee^j(t))(\psi(x)\eta(t)-H_\ee^j(t)-u_{\ee})ds{dt}-\nonumber \\
&\qquad -\frac{(n-2){\ee}^{-\gamma}}{C_{0}(1-\alpha_{\ee})}\sum\limits_{j\in \Upsilon_{\ee}}\int\limits_{0}^{T}\int\limits_{\pa{G^{j}_{\ee}}}H_\ee^j(t)(\psi(x)\eta(t)-H_\ee^j(t)-u_{\ee})ds{dt} \nonumber \\
&\qquad -{\ee}^{-\gamma}\sum\limits_{j\in \Upsilon_{\ee}}\int\limits_{0}^{T}\int\limits_{\pa{G^{j}_{\ee}}}{g(x,t)}(\psi(x)\eta(t)-H_\ee^j(t)-u_{\ee})ds{dt}
\label{eq:43}\  \\
&=\gamma_{\ee} + {\ee}^{-\gamma}\sum\limits_{j\in \Upsilon_{\ee}}\int\limits_{0}^{T}\int\limits_{\pa{G^{j}_{\ee}}} \left\{ \beta \psi(P^{j}_{\ee})\frac{d \eta}{dt}(t)- \beta \frac{d H_\ee^j}{dt}(t)+\lambda\Big(\psi(P^{j}_{\ee})\eta(t) - H_\ee^j\Big) -\frac{n-2}{C_{0}}H_\ee^j(t)-g(P^{j}_{\ee},t)\right\}\times 
\nonumber \\
&\qquad \qquad \times(\psi(x)\eta(t)-H_\ee^j(t)-u_{\ee})ds{dt}, \label{eq:reason for appearance of the strange term}
\end{align}
where $\gamma_{\ee}\to 0$ as ${\ee}\to 0$.\\

So, in conclusion, with this choice of test function
the sum of all integrals above over the boundary $S_{\ee}\times(0,T)$ tends to zero.\\

\subsection{Deduction of the effective reaction term}
From \eqref{eq:37}-\eqref{eq:43} we conclude that the function $\uo$ satisfies the integral inequality
\begin{gather*}
 \alpha \int\limits_{Q^{T}}\psi(x)\frac{d \eta}{dt}(t)(\psi(x)\eta(t)-\uo)dx{dt}+\int\limits_{Q^{T}}\nabla(\psi(x)\eta(t))\nabla(\psi(x)\eta(t)-\uo)dx{dt}-\\
-\lim\limits_{{\ee}\to 0}\sum\limits_{j\in \Upsilon_{\ee}}\int\limits_{0}^{T}\int\limits_{\pa{T^{j}_{{\ee}/4}}}\pa_{\nu}w^{j}_{\ee}H_\ee^j(t)(\psi(x)\eta(t)-u_{\ee})ds{dt}\ge\\
\ge \int\limits_{Q^{T}}f(\psi(x)\eta(t)-\uo)dx{dt}
-\frac{\alpha}{2}\Vert{\psi(x)\eta(0)}\Vert^{2}_{L^{2}(\Omega)}.
\label{eq:44}
\end{gather*}

Applying Lemma 1 from \cite{ZSh}, we deduce
\begin{gather}
\lim\limits_{{\ee}\to 0}4^{n-1}{\ee}\sum\limits_{j\in \Upsilon_{\ee}}\int\limits_{0}^{T}\int\limits_{\pa{T^{j}_{{\ee}/4}}}H_\ee^j(t)(\psi(x)\eta(t)-u_{\ee})ds{dt}= \nonumber \\
=\omega_{n}\int\limits_{0}^{T}\int\limits_{\Omega}H_{\psi\eta}(x,t)(\psi(x)\eta(t)-\uo)dx{dt},
\label{eq:45}
\end{gather} 
where $\omega_{n}$ is the area of the unit sphere in $\mathbb{R}^{n}$.
\\

\subsection{Homogenised equation for $\uo$}
Thus, we have the following integral inequality for $\uo$
\begin{align} 
	{\alpha}\int\limits_{Q^{T}}&\psi(x)\frac{d \eta}{dt}(t)(\psi(x)\eta(t)-\uo)dx{dt}+\int\limits_{Q^{T}}\nabla_{x}(\psi(x)\eta(t))\nabla(\psi(x)\eta(t)-\uo)dx{dt}+\\
&\qquad +(n-2)C_{0}^{n-2}\omega_{n}\int\limits_{Q^{T}}H_{\psi\eta}(x,t)(\psi(x)\eta(t)-\uo)dx{dt}\ge
\label{eq:47}\  \\
&\ge \int\limits_{Q^{T}}f(\psi(x)\eta(t)-\uo)dx{dt}
-\frac{\alpha}{2}
\Vert\psi(x)\eta(0)\Vert^{2}_{L^{2}(\Omega)}.
\end{align}

Taking into account that the linear span of functions $\{\psi(x)\eta(t):\psi\in C^{\infty}_{0}(\Omega), \eta\in C^{1}[0,T]\}$ are dense in the space
$$V=\left \{u\in L^{2}(0,T; H^{1}_{0}(\Omega)) \cap \mathcal C([0,T]; L^2 (\Omega) ): \pa_{t}u\in L^{2}(0,T;H^{-1}(\Omega)) \right \},$$ 
we deduce that
\begin{align} 
	{\alpha}\int\limits_{Q^{T}}&\frac{\partial \phi}{\partial t}(\phi-\uo)dx{dt}+\int\limits_{Q^{T}}\nabla_{x}\phi \nabla(\phi-\uo)dx{dt}\\
&\qquad +(n-2)C_{0}^{n-2}\omega_{n}\int\limits_{Q^{T}}H_{\phi}(x,t)(\phi-\uo)dx{dt}\ge
\\
&\ge \int\limits_{Q^{T}}f(\psi(x)\eta(t)-\uo)dx{dt}
-\frac{\alpha}{2}\Vert\psi(x)\eta(0)\Vert^{2}_{L^{2}(\Omega)},
\end{align}
for any function
$\phi\in V$. Using $\phi=\uo+\tau \varphi$ where $\varphi \in V$, we can pass to a limit as $\tau\to 0^+$ and $\tau\to 0^-$.
Due to \cref{eq:46} for $\beta > 0$ and solving \cref{eq:defn H when beta 0} explicitly when $\beta = 0$, we deduce that 
\begin{equation}
		H_{\uo + \tau \varphi} \longrightarrow H_{\uo} \qquad \textrm{ in } L^2 (Q^T)\textrm{ as } \tau \to 0.
\end{equation}
\normalcolor
We
conclude that $\uo$ is satisfying the integral identity
\begin{gather} 
	{\alpha}\int\limits_{0}^{T}\langle \pa_{t}\uo,\varphi \rangle dt+\int\limits_{Q^{T}}\nabla{\uo}\nabla{\varphi}dx{dt}
+(n-2)C_{0}^{n-2}\omega_{n}\int\limits_{Q^{T}}H_{\uo}(x,t)\varphi{dx}dt=\int\limits_{Q^{T}}f\varphi{dx}dt.
\label{eq:48}
\end{gather}
Hence, $\uo$ is a weak solution of the problem \eqref{eq:28}. 

\subsection{Comparison principle of the limit problem}
\label{sec:uniqueness of the limit problem}

Problem \eqref{eq:28} is by no means standard. However, some systems keeping several similar features was considered in the literature: see, e.g.\ \cite{DiSt} and \cite{Chipot}. 
We prove uniqueness using the change-of-variable formulation \eqref{eq:28b}. 
\begin{lemma}
	Assume that $f, g \le 0$ and let $u,v$ be a solution of \eqref{eq:28b}. Then $u,v \le 0$.
\end{lemma}

\begin{proof}

	Choosing $u_+$ as a test function in the first equation and $v_+$ we deduce that
\begin{align}
	\label{eq:comparison 1}
	\alpha \frac{d}{dt} \int_\Omega u_+^2 &+ \int_\Omega |\nabla u_+|^2 + C_1 \int_\Omega u_+^2 \le  C_1 \int_\Omega u_+ v  \le C_1 \int_\Omega u_+ v_+ \\
	\label{eq:comparison 2}
	\beta \frac{d}{dt} \int_ \Omega v_+^2 & + C_2 \int_\Omega v_+^2 = \lambda \int_\Omega u v_+ \le \lambda \int_\Omega u_+ v_+  
\end{align}
\textbf{Case 1. $\alpha, \beta > 0$.} Therefore
\begin{equation*}
	 \frac{d}{dt} \int_\Omega \left( u_+^2 + v_+^2 \right) \le \left( \frac{C_1}{\alpha} + \frac {\lambda}{\beta}  \right) \int_\Omega u_+v_+ \le \frac{1}{2} \left( \frac{C_1}{\alpha} + \frac {\lambda}{\beta}  \right) \int_\Omega(u_+^2 + v_+^2)
\end{equation*}
Since $u(0) = v(0) = 0$ we deduce, using Gronwall's inequality, that
\begin{equation*}
	u_+ = v_+ = 0 \textrm{ in }Q^T 
\end{equation*}
\textbf{Case 2. $\alpha = 0$ and $\beta > 0$.} We apply Poincaré's inequality in \eqref{eq:comparison 1} and we deduce
\begin{align}
	c_P \int_\Omega u_+^2 & \le  C_1 \int_\Omega u_+ v_+ \\
	\beta \frac{d}{dt} \int_ \Omega v_+^2 & + C_2 \int_\Omega v_+^2 \le  \lambda \int_\Omega u_+ v_+  .
\end{align}
Joining the two computations and applying Young's inequality
\begin{equation*}
	c_P \int_\Omega u_+^2 + \beta \frac{d}{dt} \int_ \Omega v_+^2 \le  C_1 \int_\Omega u_+ v_+ \le c_P \int_ \Omega u_+^2 + C_3 \int_\Omega v_+^2.
\end{equation*}
Hence, we can apply Gronwall's inequality to deduce $v_+ = 0$. Therefore, due to \eqref{eq:comparison 1}, $u_+ = 0$. \\

\noindent \textbf{Case 3. $\alpha > 0$ and $\beta = 0$} In this case we have that
\begin{equation*}
	 \frac{d}{dt} \int_\Omega  u_+^2 + C_2 \int_{\Omega}v_+^2 \le \left( \frac{C_1}{\alpha} + {\lambda}  \right) \int_\Omega u_+v_+ \le C_4 \int_\Omega u_+^2 + C_2 \int_\Omega v_+^2.
\end{equation*}
Hence, we can apply Gronwall's inequality to deduce $u_+ = 0$ and, through \eqref{eq:comparison 2}, $v_+ = 0$.
This completes the proof.
\end{proof} 

Uniqueness solutions of \eqref{eq:28} follows as an immediate consequence. 

\section*{Acknowledgments}

The research of J.I. D\'{\i}az and D. G\'{o}mez-Castro was partially supported by the project ref. MTM2017-85449-P of the DGISPI (Spain).

\end{document}